\documentclass[11pt]{amsart}
\usepackage{amsaddr, amsthm, amssymb, amsmath}
\usepackage{mathtools}
\usepackage{graphicx}
\usepackage{hyperref}
\hypersetup{colorlinks=true, linkcolor=blue, citecolor=blue, urlcolor=blue}
\usepackage{cleveref}
\usepackage{tikz}
\usepackage{url}
\usepackage{epigraph}
\usepackage{todonotes}
\usepackage{subcaption}
\usepackage{fullpage}
\usepackage[sort, round]{natbib}

\title[Pentagon Minimization without Computation]{Pentagon Minimization without Computation}
\author{John Mackey \and Bernardo Subercaseaux}
\address{Carnegie Mellon University, Pittsburgh, PA 15213}
\email{\url{jmackey@andrew.cmu.edu}, \url{bersub@cmu.edu}}

\newtheorem{theorem}{Theorem}
\newtheorem{lemma}[theorem]{Lemma}

\newcommand{\bgamma}{\bar{\gamma}}
\newcommand{\bbeta}{\bar{\beta}}

\DeclareMathOperator{\Ex}{\mathbf{E}}
\DeclareMathOperator{\Var}{\mathbf{Var}}
\DeclareMathOperator{\Cov}{\mathbf{Cov}}

\newcommand{\myCustomUSymbol}{%
\tikz[baseline=-0.75ex]{
    \coordinate (a) at (-0.5ex,-0.75ex);
    \coordinate (b) at (0.1ex, 1ex);
    \coordinate (c) at (1.8ex, 1ex);
    \coordinate (d) at (2.4ex, -0.75ex);
    \coordinate (e) at (0.95ex, -0.1ex);

    \draw[, -] (a) -- (b) -- (c) -- (d) -- cycle; 
    \draw[, -] (a) -- (c);
    \draw[, -] (b) -- (d); 

    \draw[, -] (a) -- (e) -- (d);
    \draw[, -] (b) -- (e) -- (c);

    \node[draw, circle, fill, blue, inner sep=0.1ex] (na) at (a) {};
    \node[draw, circle, fill, blue, inner sep=0.1ex] (nb) at (b) {};
    \node[draw, circle, fill, blue, inner sep=0.1ex] (nc) at (c) {};
    \node[draw, circle, fill, blue, inner sep=0.1ex] (nd) at (d) {};
    \node[draw, circle, fill, blue, inner sep=0.1ex] (ne) at (e) {};
}}

\newsavebox{\myUSymbolBox}
\savebox{\myUSymbolBox}{\myCustomUSymbol}

\newcommand{\squareOne}{\ensuremath{\usebox{\myUSymbolBox}}}

\newcommand{\myPentaSymbol}{%
\tikz[baseline=-0.75ex]{
    \coordinate (a) at (-0.6ex,-0.8ex);
    \coordinate (b) at (-0.9ex, 0.4ex); 
    \coordinate (c) at (0.35ex, 1.25ex);
    \coordinate (d) at (1.1ex, -0.8ex);
    \coordinate (e) at (1.5ex, 0.4ex);

    \draw[, -] (a) -- (b) -- (c) -- (d) -- cycle; 
    \draw[, -] (a) -- (c);
    \draw[, -] (b) -- (d); 

    \draw[, -] (a) -- (e) -- (d);
    \draw[, -] (b) -- (e) -- (c);

    \node[draw, circle, fill, blue, inner sep=0.1ex] (na) at (a) {};
    \node[draw, circle, fill, blue, inner sep=0.1ex] (nb) at (b) {};
    \node[draw, circle, fill, blue, inner sep=0.1ex] (nc) at (c) {};
    \node[draw, circle, fill, blue, inner sep=0.1ex] (nd) at (d) {};
    \node[draw, circle, fill, blue, inner sep=0.1ex] (ne) at (e) {};
}}

\newsavebox{\myPentaBox}
\savebox{\myPentaBox}{\myPentaSymbol}

\newcommand{\pentagon}{\ensuremath{\usebox{\myPentaBox}}}

\newcommand{\squareZeroSymbol}{%
\tikz[baseline=-0.75ex]{
    \coordinate (a) at (0.1ex,-0.75ex);
    \coordinate (b) at (1.85ex, -0.75ex);
    \coordinate (c) at (1.85ex, 1ex);
    \coordinate (d) at (0.1ex, 1ex);

    \draw[, -] (a) -- (b) -- (c) -- (d) -- cycle; 

    \node[draw, circle, fill, blue, inner sep=0.1ex] (na) at (a) {};
    \node[draw, circle, fill, blue, inner sep=0.1ex] (nb) at (b) {};
    \node[draw, circle, fill, blue, inner sep=0.1ex] (nc) at (c) {};
    \node[draw, circle, fill, blue, inner sep=0.1ex] (nd) at (d) {};
}}

\newsavebox{\squareZeroBox}
\savebox{\squareZeroBox}{\squareZeroSymbol}

\newcommand{\squareZero}{\ensuremath{\usebox{\squareZeroBox}}}

\newcommand{\triangleZeroSymbol}{%
\tikz[baseline=-0.75ex]{
    \coordinate (a) at (0.1ex,-0.75ex);

    \coordinate (b) at (1.8ex, -0.75ex);
    \coordinate (c) at (0.95ex, 0.6ex);

    \draw[, -] (a) -- (b) -- (c)  -- cycle; 
 
    \node[draw, circle, fill, blue, inner sep=0.1ex] (na) at (a) {};
    \node[draw, circle, fill, blue, inner sep=0.1ex] (nb) at (b) {};
    \node[draw, circle, fill, blue, inner sep=0.1ex] (nc) at (c) {};
}}

\newcommand{\triangleTwoSymbol}{%
\tikz[baseline=-0.75ex]{
    \coordinate (a) at (-0.5ex,-0.75ex);

    \coordinate (b) at (2.4ex, -0.75ex);
    \coordinate (c) at (0.95ex, 1.2ex);

    \coordinate (d) at (0.5ex, 0.0ex);
    \coordinate (e) at (1.4ex, 0.0ex);

    \draw[, -] (a) -- (b) -- (c)  -- cycle; 
    \draw[, -] (d) -- (a);
    \draw[, -] (e) -- (a);
    \draw[, -] (d) -- (b);
    \draw[, -] (e) -- (b);
    \draw[, -] (d) -- (c);
    \draw[, -] (e) -- (c);
    \draw[, -] (d) -- (e);
 
    \node[draw, circle, fill, blue, inner sep=0.1ex] (na) at (a) {};
    \node[draw, circle, fill, blue, inner sep=0.1ex] (nb) at (b) {};
    \node[draw, circle, fill, blue, inner sep=0.1ex] (nc) at (c) {};
    \node[draw, circle, fill, blue, inner sep=0.1ex] (nd) at (d) {};
    \node[draw, circle, fill, blue, inner sep=0.1ex] (ne) at (e) {};
}}

\newcommand{\triangleOneSymbol}{%
\tikz[baseline=-0.75ex]{
    \coordinate (a) at (-0.5ex,-0.75ex);

    \coordinate (b) at (2.4ex, -0.75ex);
    \coordinate (c) at (0.95ex, 1.2ex);

    \coordinate (d) at (0.9ex, 0.0ex);

    \draw[, -] (a) -- (b) -- (c)  -- cycle; 
    \draw[, -] (d) -- (a);
    \draw[, -] (d) -- (b);
    \draw[, -] (d) -- (c);
 
    \node[draw, circle, fill, blue, inner sep=0.1ex] (na) at (a) {};
    \node[draw, circle, fill, blue, inner sep=0.1ex] (nb) at (b) {};
    \node[draw, circle, fill, blue, inner sep=0.1ex] (nc) at (c) {};
    \node[draw, circle, fill, blue, inner sep=0.1ex] (nd) at (d) {};
}}

\newsavebox{\triangleTwoBox}
\savebox{\triangleTwoBox}{\triangleTwoSymbol}

\newsavebox{\triangleOneBox}
\savebox{\triangleOneBox}{\triangleOneSymbol}

\newsavebox{\triangleZeroBox}
\savebox{\triangleZeroBox}{\triangleZeroSymbol}

\newcommand{\triTwo}{\ensuremath{\usebox{\triangleTwoBox}}}
\newcommand{\triOne}{\ensuremath{\usebox{\triangleOneBox}}}
\newcommand{\triZero}{\ensuremath{\usebox{\triangleZeroBox}}}

\newlength{\bibitemsep}
\newlength{\bibparskip}
\let\oldthebibliography\thebibliography
\renewcommand\thebibliography[1]{%
  \oldthebibliography{#1}%
  \setlength{\parskip}{\bibitemsep}%
  \setlength{\itemsep}{\bibparskip}%
}

\begin{document}
\begin{abstract}
\citet{erdosCrossingNumberProblems1973}~initiated a line of research studying $\mu_k(n)$, the minimum number of convex $k$-gons one can obtain by placing $n$ points in the plane without any three of them being collinear. Asymptotically, the limits $c_k := \lim_{n\to \infty} \mu_k(n)/\binom{n}{k}$ exist for all $k$, and are strictly positive due to the Erd\H{o}s-Szekeres theorem.
This article focuses on the case $k=5$, where $c_5$ was known to be between $0.0608516$ and $0.0625$~\citep{goaoc2018limitsordertypes, subercaseaux2023minimizing}. The lower bound was obtained through the Flag Algebra method of~\citet{razborovFlagAlgebras2007}, using semi-definite programming. 
In this article we prove a more modest lower bound of $\frac{5\sqrt{5}-11}{4} \approx 0.04508$ without any computation; we exploit \emph{``planar-point equations''}  that count, in different ways, the number of convex pentagons (or other geometric objects) in a point placement.  To derive our lower bound we combine such equations by viewing them from a statistical perspective, which we believe can be fruitful for other related problems.
\end{abstract}

\maketitle

\section{Introduction}
\label{sec:intro}

\setlength{\epigraphwidth}{2.6in}
\epigraph{
    \emph{The infinite we do right away,\\ the finite might take some time.}
}{Paul Erd\H{o}s and Stanis\l{}aw Ulam,~1978.}
\vspace{6pt}

The interplay between Ramsey theory and discrete geometry was kick-started by Esther Klein, who in 1933 proved that every set of $5$ points in \emph{general position} (i.e., without any subset of three collinear points) must contain a convex quadrilateral~\citep{grahamRamseyTheory1990,Erdos1935}.
This result initiated a line of research studying $g(k)$\footnote{It is also common to use the notation $\textrm{ES}(k)$.}, the minimum number of points in general position required to guarantee the presence of a convex $k$-gon.
Let us briefly summarize what is known about $g(k)$. In their seminal paper,~\citet{Erdos1935} showed that $g(k)$ is finite for every $k$, a result often referred to as the \emph{Erd\H{o}s-Szekeres theorem}.
The first few values of the $g(k)$ function are not hard to obtain: $g(3) = 3, g(4) = 5$, and $g(5) = 9$. Based on these values, \citet{Erdos1935} conjectured that 
\[
  g(k) = 2^{k-2} + 1, \text{ for all } k \geq 3.  
\]
Then,~\citet{Erdos1961} showed through an explicit construction that the conjectured value $2^{k-2} + 1$ is a valid lower bound for $g(k)$.
On the other hand, the best upper bound at the moment is due to~\citet{holmsenTwoExtensionsErdos2020}, who following the breakthrough of~\citet{suk}, proved that 
\[
    g(k) \leq 2^{k + O(\sqrt{k \log k})}. 
\]

With respect to the values of $g(k)$ for small $k$, \citet{szekeres_peters_2006} proved that $g(6) = 17$ through computational means, and to this day, the value of $g(k)$ is not known for any $k \geq 7$.\\

A \emph{quantitative} spin-off from the $g(k)$ line of research was started by \citet{erdosCrossingNumberProblems1973}, who stated:
\begin{center} \emph{``More generally, one can ask for the least number of convex $k$-gons determined by $n$ points in the plane.''} \end{center}
Let us use 
 $\mu_k(n)$ to denote the minimum number of convex $k$-gons one can obtain by a general position placement of $n$ points in the plane.
The study of $\mu_4(n)$ has received particular attention among discrete mathematicians, since it is equivalent to $\overline{\textrm{cr}}(K_n)$, the rectilinear crossing number of the complete graph~\citep{erdosCrossingNumberProblems1973,AFMS}.
Asymptotically, the study of $\mu_k(n)$ corresponds to studying the limits 
\[ 
    c_k := \lim_{n\to \infty} \mu_k(n)/\binom{n}{k}.
\]
To see that these limits exist, consider the following folklore \emph{supersaturation} result:

\begin{lemma}[As stated by~\citet{subercaseaux2023minimizing}]\label{lemma:folklore}
	Let $m$ and $r$ be values such that $\mu_k(m) \geq r$. Then, for every $n \geq m$, we have \[\mu_k(n) \geq r \cdot \binom{n}{k}/\binom{m}{k}.\]
\end{lemma}
Using $m=n-1$ and $r= \mu_k(n-1)$, we obtain
\(
  \mu_k(n)/\binom{n}{k} \geq \mu_k(n-1)/\binom{n-1}{k}
\), 
and thus the sequence $\mu_k(n)/\binom{n}{k}$ is non-decreasing. Then, by noting the sequence $\mu_k(n)/\binom{n}{k}$ is also bounded above by $1$ we confirm the existence of the limit $c_k$.

Furthermore, from the Erd\H{o}s-Szekeres theorem, we know that $g(k)$ is finite and thus combining $\mu_k(g(k)) = 1$ with~\Cref{lemma:folklore} we obtain
\[
  c_k \geq 1/\binom{g(k)}{k} > 0.
\]
We can get a  concrete bound for $c_k$ by plugging in the upper bounds of either~\citet{holmsenTwoExtensionsErdos2020}~or~\citet{suk}, obtaining:  
\[
  c_k \geq \left(\frac{k}{2^{k+o(k)}}\right)^k.
\]

These bounds are particularly loose for small values of $k$.
For $k=4$, after a series of successive improvements, the best bounds known for $c_4$\footnote{The particular case of $c_4$ is often denoted by $q^\star$, as it is equivalent to Sylvester's \emph{four point constant}~\citep{Scheinerman_Wilf_1994}.} are due to~\citet{ABREGO2008273}~and~\citet{Ongoing}:
\[
    0.379972 \leq c_4 \leq 0.3804491869.
\]

\subsection{Previous Work}

We focus on the case $k=5$, where the best bounds known are due to~\citet{goaoc2018limitsordertypes} and~\citet{subercaseaux2023minimizing}.
\[
  0.0608516 \leq c_5 \leq 0.0625 = \frac{1}{16}.    
\]
The upper bound of~\citet{subercaseaux2023minimizing} was obtained constructively, and the authors conjectured the true value of $c_5$ to be $\frac{1}{16}$, which was also conjectured by Abrego according to~\citet{goaoc2018limitsordertypes}.
 So far, lower bounds have been obtained through two different computational methods. On one hand, by computing $\mu_5(16) = 112$ using SAT solvers, \citet{subercaseaux2023minimizing} proved that $c_5 \geq \frac{112}{4368} \approx 0.02564$, and such a result was accompanied of a MaxSAT proof that can be verified independently.
  In the same line,~\citet{Battleman} obtained $\mu_5(18) = 252$, which implies $c_5 \geq 0.029411$. 
  On the other hand, the best lower bound, due to~\citet{goaoc2018limitsordertypes}, was obtained through the Flag Algebra method of~\citet{razborovFlagAlgebras2007}, using semi-definite programming. As noted by~\citet{RasborovWhatIs}, a downside of this method is the lack of a standard proof system for semi-definite programming, which can render verification difficult. Nonetheless, it is worth noticing that~\citet{goaoc2018limitsordertypes} included scripts for verification of their results, which can be therefore reproduced in about an hour of computation.
  In terms of simplicity, despite a great exposition by~\citet{goaoc2018limitsordertypes} that contains many nice examples, the formalism of the Flag Algebra method has several prerequisites, like measure theory and finite model theory. Furthermore, the inequalities used to provide the lower bound are encapsulated in the semi-definite programming formulation, and thus it is not straightforward to deduce what is going on mathematically under the hood. This might also be considered a benefit of the method; Flag Algebras allows automating calculations that in some sense we ``do by hand''. As we argue in the next subsection, however, we believe there is value in this more manual approach.
  
\subsection{Our Contribution}

In this article, we provide a more modest lower bound by more elementary means: we use a combinatorial argument based on~\emph{``planar-point equations''}, similar to~\citep[Equation 7.]{goaoc2018limitsordertypes}, and a statistical perspective on equations that count geometric objects (cf.~\citep[Section 2.1.1]{goaoc2018limitsordertypes}). Our main result is thus:
\begin{theorem}\label{thm:main}
  For every $n > C$, where $C$ is an absolute constant, the minimum number of convex $5$-gons among $n$ points in the plane without three of them in a line, denoted by $\mu_5(n)$, satisfies
  \[
    \mu_5(n) \geq \frac{5\sqrt{5}-11}{480}n^5  \approx \frac{1}{2661}n^5,  
  \]
  and consequently,
    \[
        c_5 := \lim_{n\to \infty} \frac{\mu_5(n)}{\binom{n}{5}}\geq \frac{5\sqrt{5}-11}{4} \approx 0.04508.
    \]
\end{theorem}

Note that as pointed out by~\citet{goaoc2018limitsordertypes}, the trivial bound in this problem is~$c_5 \geq  0.00793$, so while our bound is not as strong as the one obtained by~\citet{goaoc2018limitsordertypes}, it is still far better than the trivial bound, and that the SAT-based bounds of~\citet{subercaseaux2023minimizing,Battleman}.
We believe there is value in our presentation since it shows how a clever mixture of basic combinatorial techniques can be used to obtain a non-trivial lower bound on a discrete geometry problem, and in particular, it depicts how a statistical framing of deterministic objects can greatly simplify algebraic calculations. In this sense, our article resembles the work of~\citet{BONDY199771}, which is credited as an inspiration of the Flag Algebra method~\citep{silva2016flagalgebrasglance}.
Our proof is self-contained and does not require any computation. Moreover, we believe that it can serve as a way of introducing the Flag Algebra method (in its geometric flavor) without using Flag Algebras explicitly, but rather showing what performing such an analysis manually looks like, which then allows showing the Flag Algebra method as a way to automate these calculations.

Just as how the Flag Algebra method is traditionally explained starting from Mantel's theorem~(e.g.,~\citet{silva2016flagalgebrasglance,razborovFlagAlgebras2007,RasborovWhatIs}), we can also explain our approach as a way to derive Mantel's theorem.
\begin{theorem}[Mantel, 1907]
  For every $n > 3$, the maximum number of edges in a triangle-free graph $G = (V, E)$ is $\lfloor |V|^2/4 \rfloor$.
\end{theorem}
\begin{proof}
  First, let us identify the main ``object'' of this proof: vertex-degrees. This choice is not entirely obvious, and our choices for proving~\Cref{thm:main} won't be either, but they are sensible in hindsight; the intuition is that avoiding triangles limits the degrees of vertices in a graph. 
  Second, as in our proof, double-counting over the main object is crucial to get equations that capture global properties of how the object of interest behaves. In this case, the handshake lemma suffices:
  \begin{equation}\label{eq:handshake}
    \sum_{v \in V} \deg(v) = 2|E|,
  \end{equation}
  but a statistical perspective will be algebraically helpful, and thus we reframe~\Cref{eq:handshake} as 
  \begin{equation}\label{eq:stat-handshake}
    \Ex[\deg(v)] = 2|E|/|V|,
  \end{equation}
  where the expectation is over a uniformly random choice of vertex. Then, the only use of the triangle-free property is in the following equation:
\begin{equation}\label{eq:t-free}
  \deg(u) + \deg(v) \leq n, \quad \text{for every edge } (u, v),
\end{equation}
which holds since otherwise $u$ and $v$ would have a common neighbor and thus induce a triangle.
We must now turn~\Cref{eq:t-free} into a statistical equation, which can be done by summing it over the edges:
\begin{equation*}
  \sum_{(u, v) \in E} \left(\deg(u) + \deg(v)\right) \leq n|E| \implies \sum_{v \in V} \deg(v)^2 \leq n|E| \implies \Ex[\deg(v)^2] \leq |E|. 
\end{equation*}
We now combine our statistically framed equations as:
\[
  0 \leq \Var[\deg(v)] = \Ex[\deg(v)^2] - \Ex[\deg(v)]^2 \leq |E| - (2|E|/|V|)^2,
\]
from where the result directly follows.
\end{proof}
  

\section{Planar-Point Equations}\label{sec:ppe}

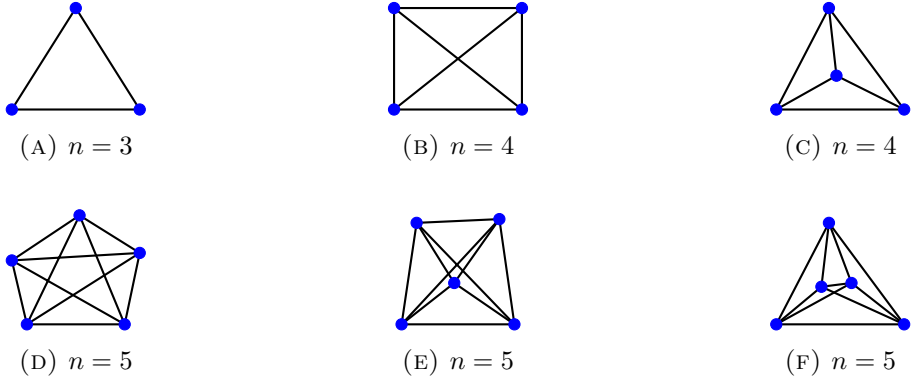
\begin{figure}
\begin{subfigure}{0.3\textwidth}
    \centering
    \begin{tikzpicture}
        \coordinate (a) at (0.1,-0.75);
        \coordinate (b) at (1.8, -0.75);
        \coordinate (c) at (0.95, 0.6);
    
        \draw[thick, -] (a) -- (b) -- (c)  -- cycle; 
    
        \node[draw, circle, fill, blue, inner sep=1.5] (na) at (a) {};
        \node[draw, circle, fill, blue, inner sep=1.5] (nb) at (b) {};
        \node[draw, circle, fill, blue, inner sep=1.5] (nc) at (c) {};
    \end{tikzpicture}
    \caption{$n=3$}
\end{subfigure}
\begin{subfigure}{0.3\textwidth}
    \centering
    \begin{tikzpicture}
        \coordinate (a) at (0.1,-0.75);
        \coordinate (b) at (1.8, -0.75);
        \coordinate (c) at (1.8, 0.6);
        \coordinate (d) at (0.1, 0.6);
    
        \draw[thick, -] (a) -- (b) -- (c) -- (d) -- cycle; 
        \draw[thick, -] (b) -- (d);
        \draw[thick, -] (a) -- (c);
    
        \node[draw, circle, fill, blue, inner sep=1.5] (na) at (a) {};
        \node[draw, circle, fill, blue, inner sep=1.5] (nb) at (b) {};
        \node[draw, circle, fill, blue, inner sep=1.5] (nc) at (c) {};
        \node[draw, circle, fill, blue, inner sep=1.5] (nd) at (d) {};
    \end{tikzpicture}
    \caption{$n=4$}
\end{subfigure}
\begin{subfigure}{0.3\textwidth}
    \centering
    \begin{tikzpicture}
        \coordinate (a) at (0.1,-0.75);
        \coordinate (b) at (1.8, -0.75);
        \coordinate (c) at (0.8, 0.6);
        \coordinate (d) at (0.9, -0.3);
    
        \draw[thick, -] (a) -- (b) -- (c)  -- cycle; 
        \draw[thick, -] (a) -- (d);
        \draw[thick, -] (b) -- (d);
        \draw[thick, -] (c) -- (d);
    
        \node[draw, circle, fill, blue, inner sep=1.5] (na) at (a) {};
        \node[draw, circle, fill, blue, inner sep=1.5] (nb) at (b) {};
        \node[draw, circle, fill, blue, inner sep=1.5] (nc) at (c) {};
        \node[draw, circle, fill, blue, inner sep=1.5] (nd) at (d) {};
    \end{tikzpicture}
    \caption{$n=4$}
\end{subfigure}

\vspace{1.5em}

\begin{subfigure}{0.3\textwidth}
    \centering
    \begin{tikzpicture}
        \coordinate (a) at (0.5,-0.75);
        \coordinate (b) at (1.8, -0.75);
        \coordinate (c) at (0.3, 0.1);
        \coordinate (d) at (2.0, 0.2);
        \coordinate (e) at (1.2, 0.7);

        \draw[thick, -] (a) -- (b) -- (c)  -- (d) -- (e) -- cycle;
        \draw[thick, -] (a) -- (d);
        \draw[thick, -] (b) -- (e);
        \draw[thick, -] (a) -- (c);
        \draw[thick, -] (b) -- (d);
        \draw[thick, -] (c) -- (e);

        \node[draw, circle, fill, blue, inner sep=1.5] (na) at (a) {};
        \node[draw, circle, fill, blue, inner sep=1.5] (nb) at (b) {};
        \node[draw, circle, fill, blue, inner sep=1.5] (nc) at (c) {};
        \node[draw, circle, fill, blue, inner sep=1.5] (nd) at (d) {};
        \node[draw, circle, fill, blue, inner sep=1.5] (ne) at (e) {};
    \end{tikzpicture}
    \caption{$n=5$}
\end{subfigure}
\begin{subfigure}{0.3\textwidth}
    \centering
    \begin{tikzpicture}
        \coordinate (a) at (0.3,-0.75);
        \coordinate (b) at (1.8, -0.75);
        \coordinate (c) at (0.5, 0.6);
        \coordinate (d) at (1.6, 0.65);
        \coordinate (e) at (1.0, -0.2);

        \draw[thick, -] (a) -- (b) -- (c)  -- (d) -- cycle;
        \draw[thick, -] (a) -- (e);
        \draw[thick, -] (b) -- (e);
        \draw[thick, -] (c) -- (e);
        \draw[thick, -] (d) -- (e);
        \draw[thick, -] (a) -- (c);
        \draw[thick, -] (b) -- (d);

        \node[draw, circle, fill, blue, inner sep=1.5] (na) at (a) {};
        \node[draw, circle, fill, blue, inner sep=1.5] (nb) at (b) {};
        \node[draw, circle, fill, blue, inner sep=1.5] (nc) at (c) {};
        \node[draw, circle, fill, blue, inner sep=1.5] (nd) at (d) {};
        \node[draw, circle, fill, blue, inner sep=1.5] (ne) at (e) {};
    \end{tikzpicture}
    \caption{$n=5$}
\end{subfigure}
\begin{subfigure}{0.3\textwidth}
    \centering
    \begin{tikzpicture}
        \coordinate (a) at (0.1,-0.75);
        \coordinate (b) at (1.8, -0.75);
        \coordinate (c) at (0.8, 0.6);
        \coordinate (d) at (0.7, -0.25);
        \coordinate (e) at (1.1, -0.2);

        \draw[thick, -] (a) -- (b) -- (c)  -- cycle;
        \draw[thick, -] (a) -- (d);
        \draw[thick, -] (b) -- (d);
        \draw[thick, -] (c) -- (d);
        \draw[thick, -] (a) -- (e);
        \draw[thick, -] (b) -- (e);
        \draw[thick, -] (c) -- (e);
        \draw[thick, -] (d) -- (e);

        \node[draw, circle, fill, blue, inner sep=1.5] (na) at (a) {};
        \node[draw, circle, fill, blue, inner sep=1.5] (nb) at (b) {};
        \node[draw, circle, fill, blue, inner sep=1.5] (nc) at (c) {};
        \node[draw, circle, fill, blue, inner sep=1.5] (nd) at (d) {};
        \node[draw, circle, fill, blue, inner sep=1.5] (ne) at (e) {};
    \end{tikzpicture}
    \caption{$n=5$}
\end{subfigure}

\caption{Illustration of the different convex hull cases for up to $5$ points. These define the \emph{types} of point placements for $n \in \{3, 4,5\}$.}\label{fig:types}
\end{figure}

Let us first introduce some notation. For the context of this article, a \emph{point placement} $P$ is a finite subset of $\mathbb{R}^2$ such that no three points in $P$ are collinear.
We will denote by $\mathcal{P}_n$ the set of all point placements of size $n$. Next, we focus on different~\emph{types} of placements in $\mathcal{P}_n$ for small values of $n$. Between $n=3$ and $n=5$, the \emph{type} of a placement $P$ can be defined by the number of points in the convex hull of $P$; as illustrated in~\Cref{fig:types}, there is a single type for $n=3$, two types for $n=4$ and three types for $n=5$.\footnote{Do not be fooled by the apparent pattern; there are $20$ \emph{types} for $n=6$ and $242$ for $n=7$. See~\cite[A006246]{oeis}.}
We avoid providing a general definition of \emph{types} as it is not currently needed for the rest of this paper; instead, we direct the interested reader to~\citet{Aichholzer_Aurenhammer_Krasser_2001,goaoc2018limitsordertypes}, or for a formal specification, to the notion of $\sigma$-equivalence in~\citet{subercaseaux2024formalverificationhexagonnumber}.
Let us remark immediately that an important difference between our work and that of~\citet{goaoc2018limitsordertypes} is that they consider types up to $n=11$ points, whereas we only consider types up to $n=5$ points to make our analysis manageable.

\begin{figure}[t]
    \centering
    \includegraphics[scale=0.75]{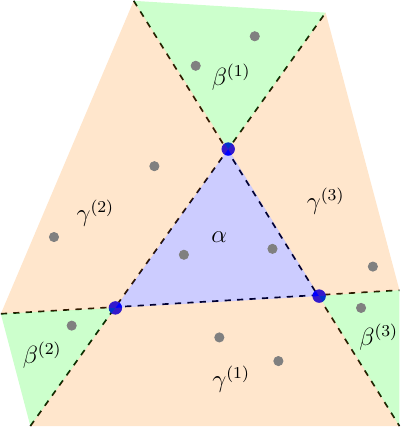}
    \caption{Illustration of the regions in a point placement induced by a given triangle, whose vertices are colored in blue.}
    \label{fig:point-configurations}
\end{figure}

For any point placement $P$, we will use $\triZero(P), \triOne(P), \squareZero(P), \squareOne(P),  \triTwo(P), \pentagon(P)$, to denote the number of subsets of $P$ of the corresponding type. 
We will prove several equations that relate these quantities, where all proofs are based on simple double-counting arguments.

Given a fixed triangle $\triZero$, we can partition the space into \emph{``regions''} with respect to $\triZero$, as depicted in~\Cref{fig:point-configurations}. Formally, given a triangle $abc$ and a point $x$, the \emph{``region''} in which $x$ is situated can be defined by cases on which half-planes $x$ belongs to (e.g., whether $x$ is ``above or below'' the lines $\vec{ab}$, $\vec{bc}$, $\vec{ac}$, which is defined by the corresponding cross products). However, we refrain from that formalization since it does not seem more helpful to a reader than~\Cref{fig:point-configurations}.

Given a fixed traingle $\triZero$, we will denote by $\beta_{\triZero}$ the number of points in the $\beta$-regions (i.e., $\beta^{(1)} \cup \beta^{(2)} \cup \beta^{(3)}$) with respect to $\triZero$ (see~\Cref{fig:point-configurations}), and by $\gamma_{\triZero}$ the number of points in the $\gamma$-regions with respect to $\triZero$. By definition, we have 
\[ 
    \beta_{\triZero} = \beta^{(1)}_{\triZero} + \beta^{(2)}_{\triZero} + \beta^{(3)}_{\triZero} \quad \text{and} \quad \gamma_{\triZero} = \gamma^{(1)}_{\triZero} + \gamma^{(2)}_{\triZero} + \gamma^{(3)}_{\triZero}.
\]
We have now the required notation to state and proof our first double-counting lemma.
\begin{lemma}\label{lemma:ppe-beta-gamma-n4}
    For any point placement $P$ in the plane, we have 
    \[
       \sum_{{\triZero} \in P} 4\beta_{\triZero}  + 3\gamma_{\triZero} = 12 \binom{n}{4}.
    \] 
\end{lemma}
\begin{proof}
Consider a fixed $\squareZero \in P$, that is, a set of $4$ points that form a convex quadrilateral.
Then, there are $4$ ways to choose a $\triZero$ such that $\triZero \subseteq \squareZero$; that is, there are $4$ triangles formed by vertices in $\squareZero$.
For each such triangle $\triZero$, the single point in $\squareZero \setminus \triZero$ will be in the $\gamma$-regions (illustrated in~\Cref{fig:point-configurations}) with respect to $\triZero$, as otherwise the fixed quadrilateral would not be convex. 
Furthermore, having fixed a triangle $\triZero$, every point in the $\gamma$-regions with respect to $\triZero$ forms a convex quadrilateral with $\triZero$. We therefore have
\begin{equation}\label{eq:-gamma-n4}
    \sum_{\triZero \in P} \gamma_{\triZero} = 4 \, \squareZero(P).
\end{equation}
From now on, we will summarize this kind of argument by saying that
\begin{center}
    \emph{``Every $\squareZero$ is counted $4$ times as a $\gamma$-configuration with respect to $\triZero$.''}
\end{center}
Similarly, we can say that 
\begin{center}
    \emph{``Every $\triOne$ is counted $3$ times as a $\beta$-configuration.''}
\end{center}

Which can be written as:
\begin{equation}\label{eq:-beta-n4}
    \sum_{\triZero \in P} \beta_{\triZero} = 3 \, \triOne(P).
\end{equation}
Note that all subsets of four points in $P$ are either a $\squareZero$ or a $\triOne$, from where
\[
    \squareZero(P) + \triOne(P) = \binom{n}{4}.
\]
By adding $3$ times~\Cref{eq:-gamma-n4} to $4$ times~\Cref{eq:-beta-n4}, we obtain the desired result, as 
\[
    \sum_{\triZero \in P} 4\beta_{\triZero}  + 3\gamma_{\triZero} = 12 \, \squareZero(P) + 12\, {\triOne}(P) = 12 \binom{n}{4}.
\]
\end{proof}

\begin{lemma}
    For any point placement $P$ in the plane, we have 
    \[ 
        4 \pentagon(P) = 4 \binom{n}{5} - \sum_{\triZero \in P} \beta_{\triZero} \gamma_{\triZero}.
    \] 
    \label{lemma:beta-gamma-count}
\end{lemma}
\begin{proof}
First, observe that
\begin{equation}\label{eq:sum-5}
    \pentagon(P) + \squareOne(P)  + \triTwo(P) = \binom{n}{5}. 
\end{equation}
It thus remains to show that 
\[
    4\left(\squareOne(P)  + \triTwo(P)\right) = \sum_{\triZero \in P} \beta_{\triZero} \gamma_{\triZero}.
\]
Indeed, if we consider a fixed $\squareOne(P)$, it will be counted $4$ times as a $\beta\gamma$-configuration (i.e., as a $\triZero$, a point in its $\beta$-regions, and another point in its $\gamma$-regions); this is illustrated in~\Cref{fig:proof-beta-gamma-count}. Similarly, if we consider a fixed $\triTwo(P)$, it will be counted $4$ times as a $\beta\gamma$-configuration, as illustrated in~\Cref{fig:proof-beta-gamma-count-2}.
The lemma follows directly by combining the two equations above.

\begin{figure}
    \begin{subfigure}{0.49\textwidth}
        \centering
    \begin{tikzpicture}[scale=0.75]
        \node[draw, circle, blue, inner sep=0.95mm, fill] (a) at (0, 0) {};
        \node[draw, circle, blue, inner sep=0.95mm, fill] (b) at (1.5, 2.5) {};
        \node[draw, circle, blue, inner sep=0.95mm, fill] (e) at (5, 0) {};
        \node[draw, circle, blue, inner sep=0.95mm, fill] (d) at (4, 2.7) {};
        \node[draw, circle, blue, inner sep=0.95mm, fill] (c) at (2.7, 1.1 ) {};
        \node[] (al) at (-0.4, -0.4) {$a$};
        \node[] (bl) at (1.1, 2.9) {$b$};
        \node[] (cl) at (2.7, 0.7) {$c$};
        \node[] (dl) at (4.3, 3.1) {$d$};
        \node[] (el) at (5.3, -0.4) {$e$};

        \draw[-, ultra thick, purple] (a) -- (b);
        \draw[-,  ultra thick, purple] (a) -- (c);
        \draw[-,  thick] (a) -- (d);
        \draw[-, thick] (a) -- (e);
        \draw[-, ultra thick, purple] (b) -- (c);
        \draw[-, thick] (b) -- (d);
        \draw[-, thick] (b) -- (e);
        \draw[-,  thick] (c) -- (d);
        \draw[-, thick] (c) -- (e);
        \draw[-, thick] (d) -- (e);
    \end{tikzpicture}
    \caption{}\label{fig:p-sub1}
    \end{subfigure}
    \begin{subfigure}{0.49\textwidth}
        \centering
    \begin{tikzpicture}[scale=0.75]
        \node[draw, circle, blue, inner sep=0.95mm, fill] (a) at (0, 0) {};
        \node[draw, circle, blue, inner sep=0.95mm, fill] (b) at (1.5, 2.5) {};
        \node[draw, circle, blue, inner sep=0.95mm, fill] (e) at (5, 0) {};
        \node[draw, circle, blue, inner sep=0.95mm, fill] (d) at (4, 2.7) {};
        \node[draw, circle, blue, inner sep=0.95mm, fill] (c) at (2.7, 1.1 ) {};
        \node[] (al) at (-0.4, -0.4) {$a$};
        \node[] (bl) at (1.1, 2.9) {$b$};
        \node[] (cl) at (2.7, 0.7) {$c$};
        \node[] (dl) at (4.3, 3.1) {$d$};
        \node[] (el) at (5.3, -0.4) {$e$};

        \draw[-, thick] (a) -- (b);
        \draw[-,  thick] (a) -- (c);
        \draw[-,  thick] (a) -- (d);
        \draw[-, thick] (a) -- (e);
        \draw[-, thick] (b) -- (c);
        \draw[-, thick] (b) -- (d);
        \draw[-, thick] (b) -- (e);
        \draw[-,  ultra thick, purple] (c) -- (d);
        \draw[-, ultra thick, purple] (c) -- (e);
        \draw[-,ultra thick, purple] (d) -- (e);
    \end{tikzpicture}
    \caption{}\label{fig:p-sub2}
\end{subfigure}

    \begin{subfigure}{0.49\textwidth}
        \centering
    \begin{tikzpicture}[scale=0.75]
        \node[draw, circle, blue, inner sep=0.95mm, fill] (a) at (0, 0) {};
        \node[draw, circle, blue, inner sep=0.95mm, fill] (b) at (1.5, 2.5) {};
        \node[draw, circle, blue, inner sep=0.95mm, fill] (e) at (5, 0) {};
        \node[draw, circle, blue, inner sep=0.95mm, fill] (d) at (4, 2.7) {};
        \node[draw, circle, blue, inner sep=0.95mm, fill] (c) at (2.7, 1.1 ) {};
        \node[] (al) at (-0.4, -0.4) {$a$};
        \node[] (bl) at (1.1, 2.9) {$b$};
        \node[] (cl) at (2.7, 0.7) {$c$};
        \node[] (dl) at (4.3, 3.1) {$d$};
        \node[] (el) at (5.3, -0.4) {$e$};

        \draw[-, thick] (a) -- (b);
        \draw[-,  ultra thick, purple] (a) -- (c);
        \draw[-,  ultra thick, purple] (a) -- (d);
        \draw[-, thick] (a) -- (e);
        \draw[-, thick] (b) -- (c);
        \draw[-, thick] (b) -- (d);
        \draw[-, thick] (b) -- (e);
        \draw[-,  ultra thick, purple] (c) -- (d);
        \draw[-, thick] (c) -- (e);
        \draw[-, thick] (d) -- (e);
    \end{tikzpicture}
    \caption{}\label{fig:p-sub3}
    \end{subfigure}
    \begin{subfigure}{0.49\textwidth}
        \centering
    \begin{tikzpicture}[scale=0.75]
        \node[draw, circle, blue, inner sep=0.95mm, fill] (a) at (0, 0) {};
        \node[draw, circle, blue, inner sep=0.95mm, fill] (b) at (1.5, 2.5) {};
        \node[draw, circle, blue, inner sep=0.95mm, fill] (e) at (5, 0) {};
        \node[draw, circle, blue, inner sep=0.95mm, fill] (d) at (4, 2.7) {};
        \node[draw, circle, blue, inner sep=0.95mm, fill] (c) at (2.7, 1.1 ) {};
        \node[] (al) at (-0.4, -0.4) {$a$};
        \node[] (bl) at (1.1, 2.9) {$b$};
        \node[] (cl) at (2.7, 0.7) {$c$};
        \node[] (dl) at (4.3, 3.1) {$d$};
        \node[] (el) at (5.3, -0.4) {$e$};

        \draw[-, thick] (a) -- (b);
        \draw[-,  thick] (a) -- (c);
        \draw[-,  thick] (a) -- (d);
        \draw[-, thick] (a) -- (e);
        \draw[-, ultra thick, purple] (b) -- (c);
        \draw[-, thick] (b) -- (d);
        \draw[-, ultra thick, purple] (b) -- (e);
        \draw[-,  thick] (c) -- (d);
        \draw[-, ultra thick, purple] (c) -- (e);
        \draw[-, thick] (d) -- (e);
    \end{tikzpicture}
    \caption{}\label{fig:p-sub4}
    \end{subfigure}

    \caption{Illustration for part of the proof of~\Cref{lemma:beta-gamma-count}, illustrating the 4 different $\triZero$ such that the remaining two points form a $\beta\gamma$-configuration. In other words, this figure illustrates how each $\squareOne$ is counted 4 times as a $\beta\gamma$-configuration.}\label{fig:proof-beta-gamma-count}
\end{figure}
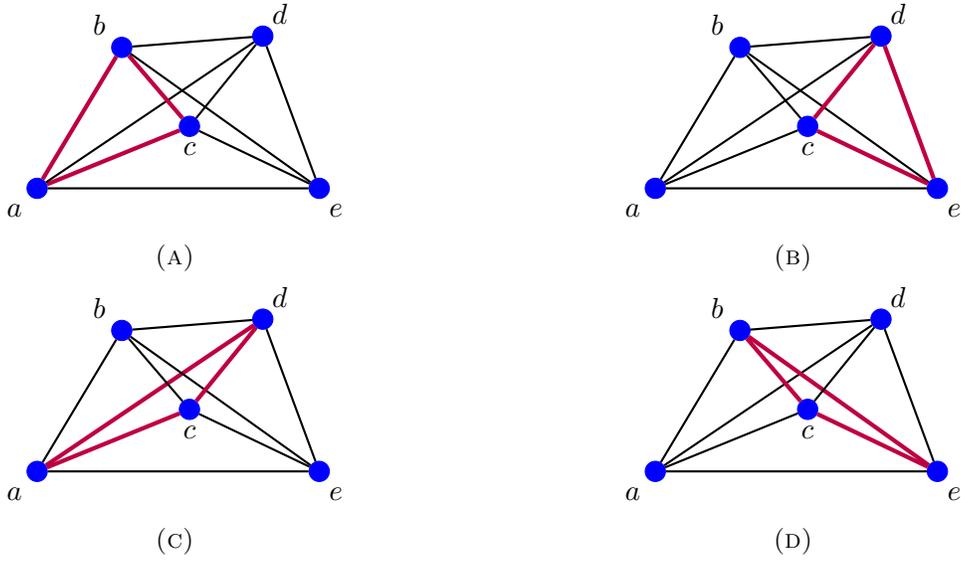

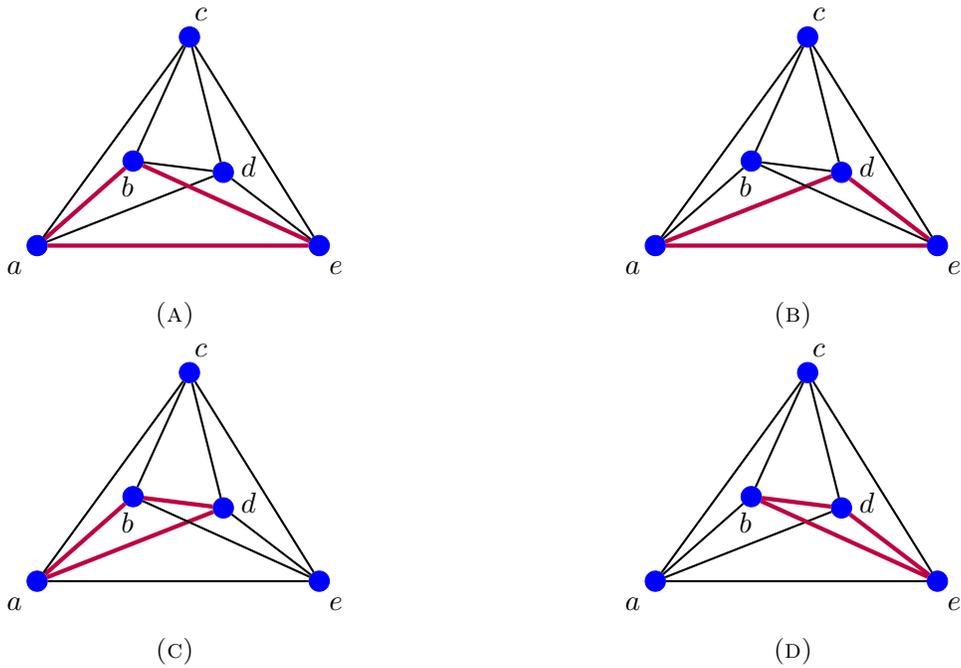
\begin{figure}
    \begin{subfigure}{0.49\textwidth}
        \centering
    \begin{tikzpicture}[scale=0.75]
        \node[draw, circle, blue, inner sep=0.95mm, fill] (a) at (0, 0) {};
        \node[draw, circle, blue, inner sep=0.95mm, fill] (b) at (1.7, 1.5) {};
        \node[draw, circle, blue, inner sep=0.95mm, fill] (e) at (5, 0) {};
        \node[draw, circle, blue, inner sep=0.95mm, fill] (c) at (2.7, 3.7) {};
        \node[draw, circle, blue, inner sep=0.95mm, fill] (d) at (3.3, 1.3 ) {};
        \node[] (al) at (-0.4, -0.4) {$a$};
        \node[] (bl) at (1.6, 1.05) {$b$};
        \node[] (cl) at (2.9, 4.1) {$c$};
        \node[] (dl) at (3.75, 1.4) {$d$};
        \node[] (el) at (5.3, -0.4) {$e$};

        \draw[-, ultra thick, purple] (a) -- (b);
        \draw[-,  thick] (a) -- (c);
        \draw[-, thick] (a) -- (d);
        \draw[-, ultra thick, purple] (a) -- (e);
        \draw[-, thick] (b) -- (c);
        \draw[-, thick] (b) -- (d);
        \draw[-, ultra thick, purple] (b) -- (e);
        \draw[-,  thick] (c) -- (d);
        \draw[-, thick] (c) -- (e);
        \draw[-, thick] (d) -- (e);
    \end{tikzpicture}
    \caption{}\label{fig:q-sub1}
    \end{subfigure}
    \begin{subfigure}{0.49\textwidth}
        \centering
    \begin{tikzpicture}[scale=0.75]
        \node[draw, circle, blue, inner sep=0.95mm, fill] (a) at (0, 0) {};
        \node[draw, circle, blue, inner sep=0.95mm, fill] (b) at (1.7, 1.5) {};
        \node[draw, circle, blue, inner sep=0.95mm, fill] (e) at (5, 0) {};
        \node[draw, circle, blue, inner sep=0.95mm, fill] (c) at (2.7, 3.7) {};
        \node[draw, circle, blue, inner sep=0.95mm, fill] (d) at (3.3, 1.3 ) {};
        \node[] (al) at (-0.4, -0.4) {$a$};
        \node[] (bl) at (1.6, 1.05) {$b$};
        \node[] (cl) at (2.9, 4.1) {$c$};
        \node[] (dl) at (3.75, 1.4) {$d$};
        \node[] (el) at (5.3, -0.4) {$e$};

        \draw[-, thick] (a) -- (b);
        \draw[-,  thick] (a) -- (c);
        \draw[-,  ultra thick, purple] (a) -- (d);
        \draw[-, ultra thick, purple] (a) -- (e);
        \draw[-, thick] (b) -- (c);
        \draw[-, thick] (b) -- (d);
        \draw[-, thick] (b) -- (e);
        \draw[-,  thick] (c) -- (d);
        \draw[-, thick] (c) -- (e);
        \draw[-, ultra thick, purple] (d) -- (e);
    \end{tikzpicture}
    \caption{}\label{fig:q-sub2}
\end{subfigure}

    \begin{subfigure}{0.49\textwidth}
        \centering
    \begin{tikzpicture}[scale=0.75]
        \node[draw, circle, blue, inner sep=0.95mm, fill] (a) at (0, 0) {};
        \node[draw, circle, blue, inner sep=0.95mm, fill] (b) at (1.7, 1.5) {};
        \node[draw, circle, blue, inner sep=0.95mm, fill] (e) at (5, 0) {};
        \node[draw, circle, blue, inner sep=0.95mm, fill] (c) at (2.7, 3.7) {};
        \node[draw, circle, blue, inner sep=0.95mm, fill] (d) at (3.3, 1.3 ) {};
        \node[] (al) at (-0.4, -0.4) {$a$};
        \node[] (bl) at (1.6, 1.05) {$b$};
        \node[] (cl) at (2.9, 4.1) {$c$};
        \node[] (dl) at (3.75, 1.4) {$d$};
        \node[] (el) at (5.3, -0.4) {$e$};

        \draw[-, ultra thick, purple] (a) -- (b);
        \draw[-,  thick] (a) -- (c);
        \draw[-,  ultra thick, purple] (a) -- (d);
        \draw[-, thick] (a) -- (e);
        \draw[-, thick] (b) -- (c);
        \draw[-, ultra thick, purple] (b) -- (d);
        \draw[-, thick] (b) -- (e);
        \draw[-,  thick] (c) -- (d);
        \draw[-, thick] (c) -- (e);
        \draw[-, thick] (d) -- (e);
    \end{tikzpicture}
    \caption{}\label{fig:q-sub3}
    \end{subfigure}
    \begin{subfigure}{0.49\textwidth}
        \centering
    \begin{tikzpicture}[scale=0.75]
        \node[draw, circle, blue, inner sep=0.95mm, fill] (a) at (0, 0) {};
        \node[draw, circle, blue, inner sep=0.95mm, fill] (b) at (1.7, 1.5) {};
        \node[draw, circle, blue, inner sep=0.95mm, fill] (e) at (5, 0) {};
        \node[draw, circle, blue, inner sep=0.95mm, fill] (c) at (2.7, 3.7) {};
        \node[draw, circle, blue, inner sep=0.95mm, fill] (d) at (3.3, 1.3 ) {};
        \node[] (al) at (-0.4, -0.4) {$a$};
        \node[] (bl) at (1.6, 1.05) {$b$};
        \node[] (cl) at (2.9, 4.1) {$c$};
        \node[] (dl) at (3.75, 1.4) {$d$};
        \node[] (el) at (5.3, -0.4) {$e$};

        \draw[-, thick] (a) -- (b);
        \draw[-,  thick] (a) -- (c);
        \draw[-,  thick] (a) -- (d);
        \draw[-, thick] (a) -- (e);
        \draw[-, thick] (b) -- (c);
        \draw[-, ultra thick, purple] (b) -- (d);
        \draw[-, ultra thick, purple] (b) -- (e);
        \draw[-,  thick] (c) -- (d);
        \draw[-, thick] (c) -- (e);
        \draw[-, ultra thick, purple] (d) -- (e);
    \end{tikzpicture}
    \caption{}\label{fig:q-sub4}
    \end{subfigure}

    \caption{Illustration for part of the proof of~\Cref{lemma:beta-gamma-count}, illustrating how a $\triTwo$ is counted 4 times as a $\beta\gamma$-configuration.}\label{fig:proof-beta-gamma-count-2}
\end{figure}

\end{proof}

\begin{lemma}\label{lemma:gamma-counts}
For any $n$-point placement $P$ in the plane, we have 
\[
    \sum_{\triZero \in P} \left(\binom{\gamma^{(1)}_{\triZero}}{2} + \binom{\gamma^{(2)}_{\triZero}}{2} + \binom{\gamma^{(3)}_{\triZero}}{2}\right) = 5\pentagon(P) + 2\squareOne(P),
\]
and 
\[
    \sum_{\triZero \in P} \left(\gamma^{(1)}_{\triZero}\gamma^{(2)}_{\triZero} + \gamma^{(2)}_{\triZero}\gamma^{(3)}_{\triZero} + \gamma^{(1)}_{\triZero}\gamma^{(3)}_{\triZero}  \right) = 5\pentagon(P) + \squareOne(P).
\]
\end{lemma}
\begin{proof}
For the first equation, we start by noting that for any $\triZero \in P$, if we pick two points from the same $\gamma$-region, we obtain either a $\pentagon$ or a $\squareOne$. Furthermore, we claim that every $\pentagon$ will be counted $5$ times this way, whereas every $\squareOne$ will only be counted twice. Indeed, every $\pentagon$ will be counted this way as long as the $\triZero$ consists of $3$ consecutive vertices of the $\pentagon$, which can be done in $5$ ways, as illustrated in~\Cref{fig:proof-gamma-counts-2}. \Cref{fig:proof-gamma-counts} illustrates how each $\squareOne$ will be counted twice. 
The second equation is obtained through an analogous counting argument: by considering the number of $\triZero$ in each fixed $\pentagon$ or $\squareOne$ such that the two remaining points are in different $\gamma$-regions.
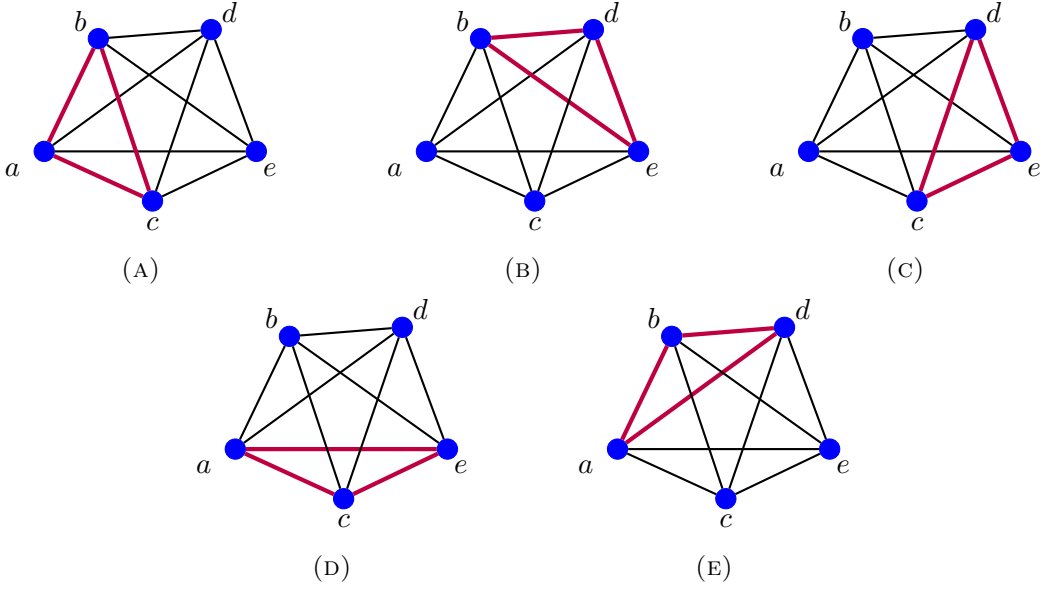
\begin{figure}
    \begin{subfigure}{0.3\textwidth}
        \centering
        \begin{tikzpicture}[scale=0.6]
            \node[draw, circle, blue, inner sep=0.95mm, fill] (a) at (0.3, 0) {};
            \node[draw, circle, blue, inner sep=0.95mm, fill] (b) at (1.5, 2.5) {};
            \node[draw, circle, blue, inner sep=0.95mm, fill] (e) at (5, 0) {};
            \node[draw, circle, blue, inner sep=0.95mm, fill] (d) at (4, 2.7) {};
            \node[draw, circle, blue, inner sep=0.95mm, fill] (c) at (2.7, -1.1 ) {};
            \node[] (al) at (-0.4, -0.4) {$a$};
            \node[] (bl) at (1.1, 2.9) {$b$};
            \node[] (cl) at (2.7, -1.6) {$c$};
            \node[] (dl) at (4.4, 3.1) {$d$};
            \node[] (el) at (5.3, -0.4) {$e$};
    
            \draw[-, ultra thick, purple] (a) -- (b);
            \draw[-, ultra thick, purple] (a) -- (c);
            \draw[-,  thick] (a) -- (d);
            \draw[-, thick] (a) -- (e);
            \draw[-, ultra thick, purple] (b) -- (c);
            \draw[-, thick] (b) -- (d);
            \draw[-, thick] (b) -- (e);
            \draw[-, thick] (c) -- (d);
            \draw[-, thick] (c) -- (e);
            \draw[-, thick] (d) -- (e);
        \end{tikzpicture}
    
    \caption{}\label{fig:r-sub1}
    \end{subfigure}
    \begin{subfigure}{0.3\textwidth}
        \centering
    \begin{tikzpicture}[scale=0.6]
        \node[draw, circle, blue, inner sep=0.95mm, fill] (a) at (0.3, 0) {};
        \node[draw, circle, blue, inner sep=0.95mm, fill] (b) at (1.5, 2.5) {};
        \node[draw, circle, blue, inner sep=0.95mm, fill] (e) at (5, 0) {};
        \node[draw, circle, blue, inner sep=0.95mm, fill] (d) at (4, 2.7) {};
        \node[draw, circle, blue, inner sep=0.95mm, fill] (c) at (2.7, -1.1 ) {};
        \node[] (al) at (-0.4, -0.4) {$a$};
        \node[] (bl) at (1.1, 2.9) {$b$};
        \node[] (cl) at (2.7, -1.6) {$c$};
        \node[] (dl) at (4.4, 3.1) {$d$};
        \node[] (el) at (5.3, -0.4) {$e$};

        \draw[-, thick] (a) -- (b);
        \draw[-, thick] (a) -- (c);
        \draw[-,  thick] (a) -- (d);
        \draw[-, thick] (a) -- (e);
        \draw[-, thick] (b) -- (c);
        \draw[-, ultra thick, purple] (b) -- (d);
        \draw[-, ultra thick, purple] (b) -- (e);
        \draw[-, thick] (c) -- (d);
        \draw[-, thick] (c) -- (e);
        \draw[-, ultra thick, purple] (d) -- (e);
    \end{tikzpicture}
    \caption{}\label{fig:r-sub2}
    \end{subfigure}
    \begin{subfigure}{0.3\textwidth}
        \centering
    \begin{tikzpicture}[scale=0.6]
        \node[draw, circle, blue, inner sep=0.95mm, fill] (a) at (0.3, 0) {};
        \node[draw, circle, blue, inner sep=0.95mm, fill] (b) at (1.5, 2.5) {};
        \node[draw, circle, blue, inner sep=0.95mm, fill] (e) at (5, 0) {};
        \node[draw, circle, blue, inner sep=0.95mm, fill] (d) at (4, 2.7) {};
        \node[draw, circle, blue, inner sep=0.95mm, fill] (c) at (2.7, -1.1 ) {};
        \node[] (al) at (-0.4, -0.4) {$a$};
        \node[] (bl) at (1.1, 2.9) {$b$};
        \node[] (cl) at (2.7, -1.6) {$c$};
        \node[] (dl) at (4.4, 3.1) {$d$};
        \node[] (el) at (5.3, -0.4) {$e$};

        \draw[-, thick] (a) -- (b);
        \draw[-, thick] (a) -- (c);
        \draw[-,  thick] (a) -- (d);
        \draw[-, thick] (a) -- (e);
        \draw[-, thick] (b) -- (c);
        \draw[-, thick] (b) -- (d);
        \draw[-, thick] (b) -- (e);
        \draw[-, ultra thick, purple] (c) -- (d);
        \draw[-, ultra thick, purple] (c) -- (e);
        \draw[-, ultra thick, purple] (d) -- (e);
    \end{tikzpicture}
    \caption{}\label{fig:r-sub3}
    \end{subfigure}

    \begin{subfigure}{0.3\textwidth}
        \centering
    \begin{tikzpicture}[scale=0.6]
        \node[draw, circle, blue, inner sep=0.95mm, fill] (a) at (0.3, 0) {};
        \node[draw, circle, blue, inner sep=0.95mm, fill] (b) at (1.5, 2.5) {};
        \node[draw, circle, blue, inner sep=0.95mm, fill] (e) at (5, 0) {};
        \node[draw, circle, blue, inner sep=0.95mm, fill] (d) at (4, 2.7) {};
        \node[draw, circle, blue, inner sep=0.95mm, fill] (c) at (2.7, -1.1 ) {};
        \node[] (al) at (-0.4, -0.4) {$a$};
        \node[] (bl) at (1.1, 2.9) {$b$};
        \node[] (cl) at (2.7, -1.6) {$c$};
        \node[] (dl) at (4.4, 3.1) {$d$};
        \node[] (el) at (5.3, -0.4) {$e$};

        \draw[-, thick] (a) -- (b);
        \draw[-, ultra thick, purple] (a) -- (c);
        \draw[-, thick] (a) -- (d);
        \draw[-, ultra thick, purple] (a) -- (e);
        \draw[-, thick] (b) -- (c);
        \draw[-, thick] (b) -- (d);
        \draw[-, thick] (b) -- (e);
        \draw[-, thick] (c) -- (d);
        \draw[-, ultra thick, purple] (c) -- (e);
        \draw[-, thick] (d) -- (e);
    \end{tikzpicture}
    \caption{}\label{fig:r-sub4}
    \end{subfigure}
    \begin{subfigure}{0.3\textwidth}
        \centering
        \begin{tikzpicture}[scale=0.6]
            \node[draw, circle, blue, inner sep=0.95mm, fill] (a) at (0.3, 0) {};
            \node[draw, circle, blue, inner sep=0.95mm, fill] (b) at (1.5, 2.5) {};
            \node[draw, circle, blue, inner sep=0.95mm, fill] (e) at (5, 0) {};
            \node[draw, circle, blue, inner sep=0.95mm, fill] (d) at (4, 2.7) {};
            \node[draw, circle, blue, inner sep=0.95mm, fill] (c) at (2.7, -1.1 ) {};
            \node[] (al) at (-0.4, -0.4) {$a$};
            \node[] (bl) at (1.1, 2.9) {$b$};
            \node[] (cl) at (2.7, -1.6) {$c$};
            \node[] (dl) at (4.4, 3.1) {$d$};
            \node[] (el) at (5.3, -0.4) {$e$};
    
            \draw[-, ultra thick, purple] (a) -- (b);
            \draw[-, thick] (a) -- (c);
            \draw[-,  ultra thick, purple] (a) -- (d);
            \draw[-, thick] (a) -- (e);
            \draw[-, thick] (b) -- (c);
            \draw[-, ultra thick, purple] (b) -- (d);
            \draw[-, thick] (b) -- (e);
            \draw[-, thick] (c) -- (d);
            \draw[-, thick] (c) -- (e);
            \draw[-, thick] (d) -- (e);
        \end{tikzpicture}
    \caption{}\label{fig:r-sub5}
    \end{subfigure}

    \caption{Illustration for part of the proof of~\Cref{lemma:gamma-counts}, depicting the five $\triZero$ for which the remaining two points are in the same $\gamma$-region. }\label{fig:proof-gamma-counts-2}
\end{figure}

\begin{figure}
    \begin{subfigure}{0.49\textwidth}
        \centering
    \begin{tikzpicture}[scale=0.75]
        \node[draw, circle, blue, inner sep=0.95mm, fill] (a) at (0, 0) {};
        \node[draw, circle, blue, inner sep=0.95mm, fill] (b) at (1.5, 2.5) {};
        \node[draw, circle, blue, inner sep=0.95mm, fill] (e) at (5, 0) {};
        \node[draw, circle, blue, inner sep=0.95mm, fill] (d) at (4, 2.7) {};
        \node[draw, circle, blue, inner sep=0.95mm, fill] (c) at (2.7, 1.1 ) {};
        \node[] (al) at (-0.4, -0.4) {$a$};
        \node[] (bl) at (1.1, 2.9) {$b$};
        \node[] (cl) at (2.7, 0.7) {$c$};
        \node[] (dl) at (4.3, 3.1) {$d$};
        \node[] (el) at (5.3, -0.4) {$e$};

        \draw[-, ultra thick, purple] (a) -- (b);
        \draw[-, thick] (a) -- (c);
        \draw[-,  ultra thick, purple] (a) -- (d);
        \draw[-, thick] (a) -- (e);
        \draw[-, thick] (b) -- (c);
        \draw[-, ultra thick, purple] (b) -- (d);
        \draw[-, thick] (b) -- (e);
        \draw[-, thick] (c) -- (d);
        \draw[-, thick] (c) -- (e);
        \draw[-, thick] (d) -- (e);
    \end{tikzpicture}
    \caption{}\label{fig:sub1}
    \end{subfigure}
    \begin{subfigure}{0.49\textwidth}
        \centering
        \begin{tikzpicture}[scale=0.75]
            \node[draw, circle, blue, inner sep=0.95mm, fill] (a) at (0, 0) {};
            \node[draw, circle, blue, inner sep=0.95mm, fill] (b) at (1.5, 2.5) {};
            \node[draw, circle, blue, inner sep=0.95mm, fill] (e) at (5, 0) {};
            \node[draw, circle, blue, inner sep=0.95mm, fill] (d) at (4, 2.7) {};
            \node[draw, circle, blue, inner sep=0.95mm, fill] (c) at (2.7, 1.1 ) {};
            \node[] (al) at (-0.4, -0.4) {$a$};
            \node[] (bl) at (1.1, 2.9) {$b$};
            \node[] (cl) at (2.7, 0.7) {$c$};
            \node[] (dl) at (4.3, 3.1) {$d$};
            \node[] (el) at (5.3, -0.4) {$e$};
    
            \draw[-, thick] (a) -- (b);
            \draw[-, thick] (a) -- (c);
            \draw[-, thick] (a) -- (d);
            \draw[-, thick] (a) -- (e);
            \draw[-, thick] (b) -- (c);
            \draw[-, ultra thick, purple] (b) -- (d);
            \draw[-,  ultra thick, purple] (b) -- (e);
            \draw[-, thick] (c) -- (d);
            \draw[-, thick] (c) -- (e);
            \draw[-,  ultra thick, purple] (d) -- (e);
        \end{tikzpicture}
        \caption{}\label{fig:sub2}
        \end{subfigure}
    \caption{Illustration for part of the proof of~\Cref{lemma:gamma-counts}, depicting the two $\triZero$ for which the remaining two points are in the same $\gamma$-region. }\label{fig:proof-gamma-counts}
\end{figure}
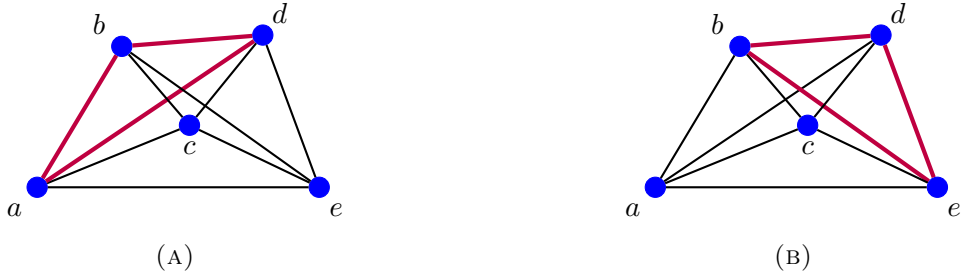
\end{proof}

\begin{lemma}\label{lemma:beta-counts}
    For any $n$-point placement $P$ in the plane, we have 
    \[
        \sum_{\triZero \in P} \left(\binom{\beta^{(1)}_{\triZero}}{2} + \binom{\beta^{(2)}_{\triZero}}{2} + \binom{\beta^{(3)}_{\triZero}}{2}\right) = \squareOne(P) + 2\triTwo(P),
    \]
    and 
    \[
        \sum_{\triZero \in P} \left(\beta^{(1)}_{\triZero}\beta^{(2)}_{\triZero} + \beta^{(2)}_{\triZero}\beta^{(3)}_{\triZero} + \beta^{(1)}_{\triZero}\beta^{(3)}_{\triZero}  \right) = \triTwo(P).
    \]
\end{lemma}
\begin{proof}
The proof of this lemma is a counting argument analogous to the proof of~\Cref{lemma:gamma-counts}.
\end{proof}

\begin{lemma}\label{lemma:c_4-counts}
For any placement  $P$ of $n$ points in the plane, we have 
\[
(n-4)\squareZero(P) = \frac{(n-4)}{4} \sum_{\triZero \in P} \gamma_{\triZero}  = 5\pentagon(P) + 3\squareOne(P) + \triTwo(P).
\]
\end{lemma}
\begin{proof}
The first equality follows from~\Cref{eq:-gamma-n4}.
The second equality requires a slight twist on the previous ideas. It is easier to prove that
\[ 
    (n-4)\squareZero(P) =  5\pentagon(P) + 3\squareOne(P) + \triTwo(P),
\]
which would complete the proof by transitivity. Indeed, we can show that both sides count the number of pairs $(\squareZero, S)$ where $S$ is a $5$-point subset of $P$ and $\squareZero \subset S$. The left-hand-side is obtained by noting that each $\squareZero$ has exactly $n-4 = |P \setminus \squareZero|$ choices for $S$. The right-hand-side partitions $S$ according to its type: every $\pentagon$ has $5$ subsets that are convex quadrilaterals, whereas every $\squareOne$ has only $3$, and every $\triTwo$ has only one. 
\end{proof}

\section{Deriving the Lower Bound}\label{sec:lower_bound}

We have deduced a variety of ``planar-point equations'', and it is now time to put them to good use. We now show how the previous equations can be combined by viewing them from a statistical perspective.  In a nutshell, we will isolate $\pentagon(P)$ in terms of the $\gamma_{\triZero}$ variables, and then minimize over the $\gamma_{\triZero}$ variables to obtain a lower bound on $\pentagon(P)$.

Before we proceed to prove a short sequence of lemmas that will yield the desired lower bound, let us introduce an algebraic trick that will greatly simplify our lives. Remember that our goal,~\Cref{thm:main}, is to prove a result that holds for a sufficiently large $n$, and ultimately to provide a bound on the limit 
\(
    0 < c_5 = \lim_{n\to \infty} \frac{\mu_5(n)}{\binom{n}{5}}.
\) In other words, we know $\mu_5(n) = \Theta(n^5)$, and our job is to find the constant hidden in the asymptotic notation $\Theta(\cdot)$. To this end, lower-order terms (e.g., $3n^4 + 7n^3 - 2n + 42$) can be safely ignored, as they will not affect the sought constant that determines $c_5$.
A concrete way to proceed according to this idea is through a minor extension of Landau notation for asymptotic equivalence: 
\[
    f(n) \sim g(n) \; \iff \; \lim_{n\to \infty} \frac{f(n)}{g(n)} = 1,
\]
which can be extended to inequalities as follows:
\[
    f(n) \lesssim g(n)  \iff  \lim_{n\to \infty} \frac{f(n)}{g(n)} \leq 1 \quad \text{ and } \quad  f(n) \gtrsim g(n)  \iff g(n) \lesssim f(n).
\]
Note that our definitions of $\lesssim$ and $\gtrsim$ are clearly transitive.
%
Furthermore, the key useful property in our use-case is that we can use this notation to absorb lower-order terms without ignoring the constant of the highest-order term. For instance,
\[ 
    3n^2 + 2n \sim 3n^2 \lesssim 3n^2 - 10 n \lesssim 4n^{3} \lesssim 5n^3 \gtrsim 2n^2,
\]
whereas  $n^2 \not\lesssim 5n \not\lesssim 4n$. A concrete case we will use repeatedly is that, for a constant $k$, we have 
\(
    \binom{n}{k} \sim \frac{n^k}{k!}.
\)
With this notation, we are finally ready to present and prove the series of lemmas that will achieve~\Cref{thm:main}.
\begin{lemma}\label{lemma:sigma_gamma}
    For any placement  $P$ of $n$ points in the plane, we have 
    \[
        32 \pentagon(P) \sim 4 \sum_{\triZero \in P} \gamma_{\triZero}^2  - 3n \sum_{\triZero \in P} \gamma_{\triZero}  + \frac{n^5}{10}. 
    \]
\end{lemma}
\begin{proof}
    We start by considering $8$ times the sum of both equations in~\Cref{lemma:gamma-counts}, which results in 
    \[ 
        8\sum_{\triZero \in P} \left(\binom{\gamma^{(1)}_{\triZero}}{2} + \binom{\gamma^{(2)}_{\triZero}}{2} + \binom{\gamma^{(3)}_{\triZero}}{2} + \gamma^{(1)}_{\triZero}\gamma^{(2)}_{\triZero} + \gamma^{(2)}_{\triZero}\gamma^{(3)}_{\triZero} + \gamma^{(1)}_{\triZero}\gamma^{(3)}_{\triZero} \right)  = 80\pentagon(P) + 24\squareOne(P).
    \]
    From the previous equation, we have 
    \[ 
        80\pentagon(P) + 24\squareOne(P) = 4\sum_{\triZero \in P} \gamma_{\triZero}^2 - 4\sum_{\triZero \in P} \gamma_{\triZero} \sim 4\sum_{\triZero \in P} \gamma_{\triZero}^2,
    \]
    Then, subtracting 12 times the second equality of~\Cref{lemma:c_4-counts} from the previous equation, we obtain
    \[ 
        20\pentagon(P) - 12\squareOne(P) - 12\triTwo(P) \sim 4\sum_{\triZero \in P} \gamma_{\triZero}^2 - (3n-8)\sum_{\triZero \in P} \gamma_{\triZero},
    \]
    and the lemma follows by adding $12$ times~\Cref{eq:sum-5}.
\end{proof}

Let us now define $\bbeta$ and $\bgamma$ as the average of $\beta_{\triZero}$ and $\gamma_{\triZero}$, respectively, over all $\triZero \in P$. That is,
\[
    \bbeta = \frac{1}{\triZero(P)} \sum_{\triZero \in P} \beta_{\triZero} \quad \text{and} \quad \bgamma = \frac{1}{\triZero(P)} \sum_{\triZero \in P} \gamma_{\triZero}.
\]

Similarly, we define 
\[ 
    \sigma_\gamma^2 = \left(\frac{1}{\triZero(P)} \sum_{\triZero \in P} \gamma_{\triZero}^2 \right)- \bgamma^2 \quad \text{and} \quad \sigma_\beta^2 =  \left(\frac{1}{\triZero(P)} \sum_{\triZero \in P} \beta_{\triZero}^2\right)- \bbeta^2.
\]
Note that $\gamma_{\triZero}^2 = \gamma_{\triZero} \cdot \gamma_{\triZero}$, which is not to be confused with $\gamma^{(2)}_{\triZero}$, the number of points in the $\gamma^{(2)}$-region. Also, we remark that the amounts $\bbeta$, $\bgamma$, $\sigma_\gamma^2$, and $\sigma_\beta^2$ are functions of the point placement $P$, but we omit an explicit reference to $P$ in order to simplify our notation, as $P$ is fixed in all our proof.
\begin{lemma}\label{lemma:960-2}
For any point placement $P$ of $n$ points in the plane, we have
\[ 
    960 \pentagon(P) \sim n^3 \left(20\sigma_\gamma^2 + 20\bgamma^2 - 15n\bgamma + 3n^2\right).
\]
\end{lemma}
\begin{proof}
We start by multiplying~\Cref{lemma:sigma_gamma} by $30$, from where we have
\[
    960\pentagon(P) \sim 120 \sum_{\triZero \in P} \gamma_{\triZero}^2  - 90n \sum_{\triZero \in P} \gamma_{\triZero}  + 3n^5.
\]
From the definitions of $\bgamma$ and $\sigma_\gamma^2$, we have
\[ 
    \sum_{\triZero \in P} \gamma_{\triZero} = \binom{n}{3} \bgamma \sim \frac{n^3}{6} \bgamma,
\]
as well as 
\[
    \sum_{\triZero \in P} \gamma_{\triZero}^2 = \binom{n}{3} \left(\sigma_\gamma^2 + \bgamma^2\right) \sim\frac{n^3}{6} \left(\sigma_\gamma^2 + \bgamma^2\right).
\]
Combining these three equations we obtain 
\begin{equation*}
    960\pentagon(P) \sim 120 \cdot \frac{n^3}{6} \left(\sigma_\gamma^2 + \bgamma^2\right) - 90n \cdot \frac{n^3}{6} \bgamma + 3n^{5}
\end{equation*}
from where the lemma follows.
\end{proof}

\begin{lemma}\label{lemma:sigma-beta-counts}
    For any point placement $P$ of $n$ points in the plane, we have
    \[ 
        960 \pentagon(P) \sim n^3 \left(80\sigma_\beta^2 + 45\bgamma^2 - 50n\bgamma + 13n^2\right).
    \]
\end{lemma}
\begin{proof}
Start by considering $32$ times the sum of both equations in~\Cref{lemma:beta-counts}, which yields 
\[ 
    32 \sum_{\triZero \in P} \left(\binom{\beta^{(1)}_{\triZero}}{2} + \binom{\beta^{(2)}_{\triZero}}{2} + \binom{\beta^{(3)}_{\triZero}}{2} + \beta^{(1)}_{\triZero}\beta^{(2)}_{\triZero} + \beta^{(2)}_{\triZero}\beta^{(3)}_{\triZero} + \beta^{(1)}_{\triZero}\beta^{(3)}_{\triZero} \right)  = 32\squareOne(P) + 96\triTwo(P),
\]
from where 
\[ 
    16\sum_{\triZero \in P} \beta^2_{\triZero} - 16\sum_{\triZero \in P} \beta_{\triZero} = 32\squareOne(P) + 96\triTwo(P), 
\]
and thus 
\begin{equation}\label{eq:beta-gamma-n4}
    16\sum_{\triZero \in P} \beta^2_{\triZero} - \frac{16}{6}n^3\bbeta \sim 32\squareOne(P) + 96\triTwo(P).
\end{equation}

Dividing~\Cref{lemma:ppe-beta-gamma-n4} by $\triZero(P) = \binom{n}{3}$, we obtain
\[
    4\bbeta + 3\bgamma = 12\binom{n}{4}/\binom{n}{3} = \frac{12}{4}(n-3) \sim 3n,
\]
from where we have
\begin{equation}\label{eq:beta-gamma-n}
    n  \sim \bgamma + \frac{4}{3}\bbeta.
\end{equation}

By substituting $\bbeta$ according to~\Cref{eq:beta-gamma-n} in~\Cref{eq:beta-gamma-n4}, we have 
\begin{align*}
    32\squareOne(P) + 96\triTwo(P) &\sim 16\sum_{\triZero \in P} \beta^2_{\triZero} - \frac{16}{6}n^3\bbeta\\
    &\sim 16\sum_{\triZero \in P} \beta^2_{\triZero} - 2n^3(n - \bgamma).
\end{align*}
To the previous equation we add 32 times the second equality of~\Cref{lemma:c_4-counts}, obtaining 
\begin{align}
    160\pentagon(P) + 128\squareOne(P) + 128\triTwo(P) &\sim 16\sum_{\triZero \in P} \beta^2_{\triZero} - 2n^3(n - \bgamma) + \frac{32n^4}{24}\bgamma \label{eq:sum-8-3-1}\\
    &\sim 16\sum_{\triZero \in P} \beta^2_{\triZero}  + \frac{4}{3}n^4  \bgamma \label{eq:sum-8-3-2}\\
        &\sim \frac{16}{6}n^3(\sigma_\beta^2 + \bbeta^2) + \frac{4}{3}n^4\bgamma\\
        &\sim \frac{8}{3}n^3\left(\sigma_\beta^2 + \left(\frac{3}{4}n - \frac{3}{4}\bgamma\right)^2\right) + \frac{4}{3}n^4\bgamma\\
        &\sim \frac{8}{3}n^3\sigma_\beta^2 + \frac{3}{2}n^5- \frac{5}{3}n^4\bgamma +\frac{3}{2}n^3\bgamma^2. \label{eq:sum-8-3}
\end{align}
    Note that going from~\Cref{eq:sum-8-3-1} to~\Cref{eq:sum-8-3-2} we got rid of the term $2n^3(n - \bgamma)$, since $2n^3(n - \bgamma) \in O(n^4)$, whereas the left-hand-side is $\Theta(n^5)$ since already $\pentagon(P)$ is $\Theta(n^5)$.
    Then, to~\Cref{eq:sum-8-3} we subtract $128$ times~\Cref{eq:sum-5}, thus obtaining 
    \begin{align*}
    32\pentagon(P) &\sim \frac{8}{3}n^3\sigma_\beta^2 + \frac{3}{2}n^5 -\frac{128}{120}n^5- 3n^4\bgamma +\frac{3}{2}n^3\bgamma^2\\
        &\sim \frac{8}{3}n^3\sigma_\beta^2 + \frac{13}{30}n^5- \frac{5}{3} n^4\bgamma +\frac{3}{2}n^3\bgamma^2,
    \end{align*}
    from where the lemma follows after multiplying by $30$.
\end{proof}
We will need as well the following lemma, whose proof uses a statistical inequality.
\begin{lemma}\label{eq:beta-gamma-count}
    For every point placement $P$ we have 
    \[
        960 \pentagon(P) \gtrsim n^3 \left(-30n\bgamma + 30\bgamma^2 -40\sigma_\gamma\sigma_\beta+ 8n^2\right). 
    \]
\end{lemma}
\begin{proof}
Multiplying~\Cref{lemma:beta-gamma-count} by $8$, we obtain 
\begin{align*}
    32\pentagon(P) &= 32 \binom{n}{5} - 8\sum_{\triZero \in P} \beta_{\triZero} \gamma_{\triZero} \\
            &\sim \frac{4n^5}{15} -  8\sum_{\triZero \in P} \beta_{\triZero} \gamma_{\triZero}.
\end{align*}
Now, let us study the ``covariance'' of $\beta_{\triZero}$ and $\gamma_{\triZero}$, defined as 
\[
    \Cov({\beta_{\triZero}, \gamma_{\triZero}}) = \frac{\sum_{\triZero \in P} (\beta_{\triZero} - \bbeta)(\gamma_{\triZero} - \bgamma)}{\triZero(P)} = \frac{\sum_{\triZero \in P} (\beta_{\triZero} - \bbeta)(\gamma_{\triZero} - \bgamma)}{\binom{n}{3}}.
\]
From the previous equation, we have
\[
   \binom{n}{3} \Cov({\beta_{\triZero}, \gamma_{\triZero}}) = \sum_{\triZero \in P} \beta_{\triZero} \gamma_{\triZero} - \beta_{\triZero} \bgamma - \gamma_{\triZero} \bbeta  + \bbeta \bgamma.
\]
By the covariance inequality (which is a direct consequence of Cauchy-Schwarz), we have
\[
    \Cov({\beta_{\triZero}, \gamma_{\triZero}}) \leq \sigma_\beta \sigma_\gamma,
\]
which together with the previous equation gives
\begin{align*}
    \sum_{\triZero \in P} \beta_{\triZero} \gamma_{\triZero} &\leq \binom{n}{3} \left(\sigma_\beta \sigma_\gamma \right) + \sum_{\triZero \in P}\beta_{\triZero} \bgamma + \gamma_{\triZero} \bbeta  - \bbeta\bgamma\\
    &= \binom{n}{3} \left(\sigma_\beta \sigma_\gamma  + \bgamma \bbeta\right). 
\end{align*}
We thus have the following chain of inequalities:
\begin{align*}
    32\pentagon(P) &\sim \frac{4n^5}{15} -  8\sum_{\triZero \in P} \beta_{\triZero} \gamma_{\triZero}\\
    &\geq \frac{4n^5}{15} - 8\binom{n}{3} \left(\sigma_\beta \sigma_\gamma  + \bgamma \bbeta\right)\\
    &\sim \frac{4n^5}{15} - \frac{4n^3}{3} \sigma_\gamma \sigma_\beta - \frac{4n^3}{3} \bgamma \bbeta\\
    &\sim \frac{4n^5}{15} - \frac{4n^3}{3} \sigma_\gamma \sigma_\beta - n^3 \bgamma \left(n - \bgamma\right) \tag{\Cref{eq:beta-gamma-n}},\\
    &\sim \frac{4n^5}{15} - \frac{4n^3}{3} \sigma_\gamma \sigma_\beta - n^4\bgamma + n^3\bgamma^2,
\end{align*}
from where the lemma follows after multiplying by $30$ once again.
\end{proof}

We are now ready to prove the key lemma, from which the final result will be derived.

\begin{lemma}\label{lemma:960}
Let $x_P := 960 \pentagon(P) / n^3$. Then 
\[
    x_P \gtrsim \frac{n\left(25 \bgamma^2 -22n\bgamma + 5n^2\right)}{\bgamma} \geq (10\sqrt{5}-22)n^2.    
\]
\end{lemma}
\begin{proof}
   We will use the following three equations:
   \begin{equation}
        20 \sigma^2_\gamma \sim x_P - 20\bgamma^2 + 15n\bgamma - 3n^2. \tag{\Cref{lemma:960-2}}
   \end{equation}
   \begin{equation}
        80 \sigma^2_\beta \sim x_P - 45\bgamma^2 + 50n\bgamma - 13n^2. \tag{\Cref{lemma:sigma-beta-counts}}
   \end{equation}
    \begin{equation}
        40 \sigma_\gamma \sigma_\beta \gtrsim 30 \bgamma^2 - 30n\bgamma + 8n^2 - x_P. \tag{\Cref{eq:beta-gamma-count}}
   \end{equation}
   Noticing that $(20 \sigma^2_\gamma) \cdot (80 \sigma^2_\beta) ={(40 \sigma_\gamma \sigma_\beta)}^2$, we have
   \[
   \begin{multlined}
   \left( x_P - 20\bgamma^2 + 15n\bgamma - 3n^2\right)
     \cdot  \left(x_P - 45\bgamma^2 + 50n\bgamma - 13n^2\right) \\ \gtrsim {(30 \bgamma^2 - 30n\bgamma + 8n^2 - x_P)}^2,
   \end{multlined}
   \]
   which after carrying out the multiplications simplifies to 
   \begin{equation}
    x_P \left(\bgamma n - \bgamma^2 \right) \gtrsim -25\bgamma^3n + 47\bgamma^2n^2 - 27\bgamma n^3 + 5n^4.
   \end{equation}
    We now proceed to factor the right-hand side as follows:
   \begin{align*}
    x_P \left(\bgamma n - \bgamma^2 \right) &\gtrsim n\left(-25\bgamma^3+47\bgamma^2n - 27\bgamma n^2 + 5n^3\right)\\
    &= n\left([-25\bgamma^3 + 25\bgamma^2n] + [22\bgamma^2n - 22\bgamma n^2] + [-5\bgamma n^2 + 5n^3]\right)\\ 
    &= n \left(\bgamma n - \bgamma^2 \right)\left([25\bgamma] + [-22n] + [5n^2/\bgamma]\right).
   \end{align*}
   As $\bgamma < n$, we can safely divide both sides by the positive amount $\left(\bgamma n - \bgamma^2\right)$, obtaining
   \[
    x_P \gtrsim n\left(25\bgamma - 22n + \frac{5n^2}{\bgamma}\right) = \frac{n\left(25 \bgamma^2 -22n\bgamma + 5n^2\right)}{\bgamma}.
   \]
   For the second inequality, consider that
    \begin{align*}
        \frac{n\left(25 \bgamma^2 -22n\bgamma + 5n^2\right)}{\bgamma} &= 25n\bgamma - 22n^2 + \frac{5n^3}{\bgamma}\\
        &= 25n\left(\bgamma + \frac{n^2}{5\bgamma}\right) - 22n^2,
    \end{align*}
    which, given $\bgamma \geq 0$, is clearly minimized when $\bgamma = \frac{n^2}{5\bgamma}$, i.e., $\bgamma = \frac{n}{\sqrt{5}} = \frac{\sqrt{5}}{5}n$. Therefore, we have
    \[
        x_P \gtrsim \frac{n\left(25 \bgamma^2 -22n\bgamma + 5n^2\right)}{\bgamma} \geq 25n \left(\frac{\sqrt{5}}{5}n + \frac{n^2}{\sqrt{5}n}\right) - 22n^2 = (10\sqrt{5}-22)n^2.
    \]
\end{proof}

Our main result is now a direct consequence of~\Cref{lemma:960} as we show next.
\begin{proof}[Proof of~\Cref{thm:main}]
    For any $n$, let $P^\star_n$ be an $n$-point placement that minimizes the number of convex pentagons. In other words, $\pentagon(P^\star_n) = \mu_5(n)$. 
     Using notation $x(P) := 960 \pentagon(P)/n^3$ as in~\Cref{lemma:960}, we have $x(P^\star_n) = 960 \mu_5(n)/n^3$, from where $\mu_5(n) = n^3/960 \cdot x(P^\star_n)$. 
    Therefore, using~\Cref{lemma:960} we have
    \begin{equation}\label{eq:last}
        \mu_5(n) = n^3/960 \cdot x(P^\star_n) \gtrsim n^3/960 \cdot (10\sqrt{5}-22) n^2 = \frac{5\sqrt{5}-11}{480}n^5.
    \end{equation}
    We can therefore conclude as well that
    \begin{align*}
        c_5 &= \lim_{n \to \infty} \mu_5(n)/\binom{n}{5}\\ 
            &=  \lim_{n \to \infty} \frac{\mu_5(n)}{n^5} \cdot \frac{n^5}{\binom{n}{5}}\\
            &= \lim_{n \to \infty} \frac{\mu_5(n)}{n^5} \cdot \lim_{n \to \infty} \frac{n^5}{\binom{n}{5}}\\
            &\geq \frac{5\sqrt{5}-11}{480} \cdot 120 \tag{\Cref{eq:last} and definition of $\gtrsim$}\\
            &= \frac{5\sqrt{5}-11}{4}.
    \end{align*}

\end{proof}

\bibliographystyle{unsrtnat}
\bibliography{references}
\end{document}